\documentclass[11pt]{article}
\usepackage[T1]{fontenc}
\usepackage{amsfonts}
\usepackage{amsmath}
\usepackage{amssymb}
\usepackage{amsthm}
\usepackage{bbm}
\usepackage{bm}
\usepackage{mathrsfs}
\usepackage{verbatim}
\usepackage{color}
\usepackage{pdfsync}
\usepackage{enumerate}

\theoremstyle{plain}
\newtheorem{theorem}{Theorem}

\newtheorem{lemma}[theorem]{Lemma}

\theoremstyle{definition}
\newtheorem{definition}{Definition}

\newtheorem*{remarks}{Remarks}
\theoremstyle{remark}

\newcommand{\RR}{\mathbb{R}}
\newcommand{\EE}{\mathbb{E}}
\newcommand{\PP}{\mathbb{P}}
\newcommand{\NN}{\mathbb{N}}

\newcommand{\eps}{\varepsilon}

\usepackage[margin=31 mm]{geometry}
\newcommand{\mykill}[1]{}

\usepackage[pdfborder={0 0 0}]{hyperref}
\hypersetup{
  urlcolor = black,
  pdfauthor = {Phillip Protter, Andres Riveros Valdevenito}
  pdfkeywords = {Totally Inaccessible Stopping Times, Feller Processes, Cox Construction, Compensators },
  pdftitle = {Markov Process Jump Times and Their Cox Construction},
  pdfsubject = {Markov Process Jump Times and Their Cox Construction},
  pdfpagemode = UseNone
}
\usepackage[capitalize]{cleveref}

\begin{document}

\title{\vspace{-1em}
Markov Process Jump Times and Their Cox Construction}
\date{\today}
\author{
  Philip Protter%
  \thanks{
  Columbia University, pep2117@columbia.edu. Supported in part by NSF grant DMS-2106433}
  \and
  Andr{\'e}s Riveros Valdevenito%
  \thanks{Department of Statistics, Columbia University, ar4151@columbia.edu.}
  }
  
\maketitle \vspace{-1.2em}

\begin{abstract}
In this short paper, we connect the procedure of constructing a totally inaccessible stopping time for a given process using the well-known Cox construction, dependent on an independent exponential random variable; with naturally occurring jump times of Feller processes. Ultimately, we show that these two phenomena are not only related, but are in fact two examples of the same object. We link this fact to the behaviour of predictable stopping times, by proving that they can always be written as the hitting time of zero of a continuous process. 
\end{abstract}

\vspace{.3em}

{\small

\vspace{.3em}

\maketitle

\section{Introduction}

\hspace{\parindent} Understanding and predicting random events is completely intertwined with the study of stopping times and their behaviour. To facilitate the readers comprehension, all of the standard mathematical terms not defined in what follows can be found in \cite{Protter.05}. The setup for modelling stopping times is virtually always the same. Given a filtered probability space $\left( \Omega, \mathcal{F}, \left(\mathcal{F}_{t} \right)_{t \geq 0}, \PP \right)$ satisfying the usual conditions, a stopping time is a positive random variable $\tau: \Omega \rightarrow \RR_{+}$ such that, for any $t \in [0, \infty)$, it satisfies
\begin{align*}
	\left\{\tau \leq t \right\} \in \mathcal{F}_{t}
\end{align*}

However, random events, by nature, are hard to directly predict. Hence, it is common to approach capturing elements of $\tau$'s behaviour by studying its compensator, defined as follows (see \cite{Protter.05}, III.5). Given a stopping time $\tau$, the stochastic process $ \left( 1_{t \geq \tau} \right)_{t \geq 0}$ is positive, increasing, bounded, right-continuous and adapted. Thus, by the Doob-Meyer decomposition, there exists an increasing, right-continuous, predictable process $\left(A(t)\right)_{t \geq 0}$, satisfying $A(0) = 0$,  such that
\begin{align}
	\left( 1_{t \geq \tau}  - A(t) \right)_{t \geq 0} \label{compdef}
\end{align}
is a martingale. The process $A$ is called the compensator of $\tau$ for the filtration $\left( \mathcal{F}_{t} \right)_{t \geq 0}$ and it is unique. 
\\

More precisely, the fact that \eqref{compdef} is a martingale and that $ \left( 1_{t \geq \tau} \right)_{t \geq 0}$ is constant on the interval $[\tau, \infty)$ indicates that $A(t)$ can be written as $A(t \wedge \tau)$. Therefore, we have that \eqref{compdef} is equivalent to
\begin{align}
	\left( 1_{t \geq \tau}  - A(t \wedge \tau) \right)_{t \geq 0} \label{compdef2}
\end{align}
being a martingale.  Notice that if the stopping time $\tau$ is predictable (see \cite{Protter.05}, III.2), then the compensator for the process $ \left( 1_{t \geq \tau} \right)_{t \geq 0}$ is just the process itself, it being predictable and the compensator unique, and so this decomposition is useful in the case where $\tau$ is not predictable.
\\

Conversely, given an increasing, right-continuous, predictable process $\left(A(t)\right)_{t \geq 0}$, satisfying $A(0) = 0$, on a filtered probability space satisfying the usual conditions; a stopping time can be constructed by using the now standard Cox construction. Assuming the probability space is sufficiently rich to contain an exponential random variable $Z$, with parameter 1, that is independent of the underlying filtration to which $A$ is adapted to, we can define
\begin{align}
	\tau = \inf \left\{ t \geq 0 \text{ : } A(t) \geq Z \right\} \label{coxtimedef}
\end{align}

It is then straightforward to see that $\tau$ defined as in \eqref{coxtimedef} is a stopping time with compensator $\left( A(\tau \wedge t) \right)_{t \geq 0} $, albeit for a larger filtration, where we include the information given by the random variable $Z$ in a progressive manner (see \cite{Protter.05} IV.3). 
\\

Several work has been done in both studying and applying this construction to a variety of financial problems. A significant amount of these studies (see \cite{GueyeJeanblanc.22}, \cite{GuoZeng.08} and \cite{JansonM'BayeProtter.11} for examples) focus on the conditions necessary so that  $\left(A(t)\right)_{t \geq 0}$ admits the following form
\begin{align*}
	A(t) = \int_{0}^{t} \alpha_{s} ds
\end{align*}
for a process $\left( \alpha_{s} \right)_{s \geq 0}$ satisfying myriad properties. This setup is particularly interesting when studying default times in financial markets (see \cite{Lando.98}, \cite{Lando.04} and \cite{ProtterQuintos.21} for examples), but it is restrictive and ultimately not the focus of this paper. 
\\

The question we seek to answer comes as a consequence of a pivotal theorem of P.A. Meyer (see \cref{meyerstheorem}) that characterizes the nature of positive stopping times of a strong Markov Feller process and, in particular, concludes that the jump times of such a process are totally inaccessible . Thus, if we consider the natural filtration of the Markov process (see \cite{Protter.05}, I.5 Definition 1), we have the existence of totally inaccessible stopping times (see \cite{Protter.05}, III.2) without the need of including the external exponential random variable $Z$ we use in constructions of the type of \eqref{coxtimedef}. This begs the question, is $Z$ needed as part of the construction? And if it is, where is this exponential random variable in the jump times of strong Markov Feller processes? 


\section{Results}

\hspace{\parindent}We begin by defining the class of Markov processes that will be the focus of our work in this paper. We will use the definition and notation of \cite{ChungWalsh.05}, chapter 2.

\begin{definition}\label{fellerprocess} A Markov process $\left( X_{t} \right)_{t \geq 0}$, with transition semigroup $\left( P_{t} \right)_{t \geq 0}$ is Feller if
\begin{enumerate}[1.]
	\item \label{fellerone} $P_{0}$ is the identity mapping.
	\item \label{fellertwo} For all $f \in \mathcal{C}_{0}$, $P_{t}f \in \mathcal{C}_{0}$, for all $t \geq 0$, where $\mathcal{C}_{0}$ are all continuous functions vanishing at infinity.
	\item \label{fellerthree} For all $f \in \mathcal{C}_{0}$, for all $x \in \RR$ 
	\begin{align*}
		\lim_{t \downarrow 0}P_{t}f(x) = f(x) 
	\end{align*}
\end{enumerate}
where we have that $P_{t}f(x) = \EE(f(X_{t} + x))$.
\end{definition}

We focus on Feller processes because they are the quintessential example of general strong Markov processes (see \cite{Protter.05}, I.5) defined by properties of their semigroup. The results in this paper may apply to more general processes, but we limit our results to Feller processes to skip dealing with possible complications of further generality and, moreover, they turn out to be fundamental in their connection with totally inaccessible stopping times.
\\

The next step is stating the result by P.A. Meyer that originates the discussion and results present in this paper.

\begin{theorem}[\textbf{Meyer's Theorem} (see \cite{Protter.05}, III.2)]\label{meyerstheorem} Let $(X_{t})_{t \geq 0}$ be a (strong) Markov Feller process for the probability $\PP^{\mu}$, where $X_{0}$ is distributed by $\mu$, and consider it's natural completed filtration $\mathbb{F}^{\mu}$. Let $T$ be a stopping time such that $\PP^{\mu}(T > 0) = 1$. Define
\begin{align} \label{meyersset}
	\Lambda := \left\{\omega \in \Omega \text{  :  }  X_{T(\omega)} \neq X_{T^{-}(\omega)} \text{ , } T(\omega) < \infty  \right\}
\end{align}
where $X_{T^{-}(\omega)} = \lim_{t \uparrow T(\omega)} X_{t}$. Then, $T$ satisfies that $T = T_{\Lambda} \wedge T_{\Lambda^{c}}$, where $T_{\Lambda}$ is totally inaccessible and $T_{\Lambda^{c}}$ is predictable; where, for any set $A \in \mathcal{F}_{T}$,
\begin{align*}
	T_{A}(\omega) = \begin{cases} T(\omega) & \text{ if } \omega \in A \\
	\infty & \text{ if } \omega \notin A
	\end{cases}
\end{align*} 
\end{theorem}

The implications of this theorem are many, nonetheless, for our purposes, we use it in the following way. Suppose we have a Feller process $(X_{t})_{t \geq 0}$ with a fixed initial condition $X_{0} = x$, under its natural filtration; and let $\tau$ be a jump time of $(X_{t})_{t \geq 0}$. Thus $\tau$ is a stopping time (this is easy to see), and moreover, property \ref{fellerone} of $(P_{t})_{t \geq 0}$ assures us that $\PP^{x}(\tau > 0) = 1$, where $\PP^{x}$ is just the law of the process starting at $x$.
\\

Therefore, we can apply \cref{meyerstheorem} to $\tau$ to obtain that, because $\tau$ is a jump time of $(X_{t})_{t \geq 0}$, then $\tau$ must be totally inaccessible. This is because, for a jump time, we have that the set $\Lambda$ defined in \eqref{meyersset} is equal to the set $\left\{T < \infty \right\}$. Notice that this probability space thus admits a totally inaccessible stopping time with seemingly no additional independent exponential random variables needed, like those present in \eqref{coxtimedef}. 
\\

The existence of this totally inaccessible stopping time therefore implies, by the Doob-Meyer decomposition, the existence of an increasing, right-continuous, predictable process $\left(A(t)\right)_{t \geq 0}$, satisfying $A(0) = 0$,  such that
\begin{align*}
	\left( 1_{t \geq \tau}  - A(t \wedge \tau) \right)_{t \geq 0}
\end{align*}
is a martingale, as in \eqref{compdef2}. Moreover, it is well known (see, for example, \cite{Protter.05}, III.5, Theorem 24) that, because $\tau$ is totally inaccessible, then $A$ has to be continuous almost surely. 
\\

With all this framework, we can now state the main theorem of this paper, which unifies the Cox construction of totally inaccessible stopping times by enlarging a filtration with the existence of totally inaccessible stopping times in given natural filtrations.

\begin{theorem} Given all the previous definitions, we have that the random variable $A(\tau)$ follows an exponential distribution with parameter 1, for all $\tau < \infty$ a.s.
\end{theorem} 

\begin{proof} To fix ideas, let us first further assume that $\left(A(t)\right)_{t \geq 0}$ is strictly increasing. This allows us to define it's inverse $A^{-1}(\cdot)$ for every $t$, by continuity of $A$, when $t \leq \tau$, and extend it naturally to $A^{-1}(t) = \infty$ when $t > \tau$. Then, if we define $s = A(t)$ and then apply this time change onto \eqref{compdef2}, we obtain
\begin{align*}
	1_{A^{-1}(s) \geq \tau}  - A(A^{-1}(s) \wedge \tau) & = 1_{A^{-1}(s) \geq \tau}  - A(A^{-1}(s)) \wedge A(\tau) \\
	& = 1_{s \geq A(\tau)}  - s \wedge A(\tau)
\end{align*}
where we used the fact that $A$ is increasing. This then means that a.s., for any $s > u \geq 0$, we have
\begin{align} 
	\EE\left(1_{s \geq A(\tau)}  - s \wedge A(\tau)|\mathcal{F}_{A^{-1}(u)} \right) & = \EE\left(1_{A^{-1}(s) \geq \tau}  - A(A^{-1}(s) \wedge \tau)|\mathcal{F}_{A^{-1}(u)} \right) \nonumber \\
	& = 1_{A^{-1}(u) \geq \tau}  - A(A^{-1}(u) \wedge \tau) \label{martingaleprop} \\
	& = 1_{u \geq A(\tau)}  - u \wedge A(\tau) \nonumber
\end{align}
using the fact that \eqref{compdef2} is a martingale, we are using the natural filtration, and both $A^{-1}(s)$ and $A^{-1}(u)$ can be seen as stopping times, because they are hitting times of a predictable process. This proves that
\begin{align*}
	\left( 1_{t \geq A(\tau)}  - t \wedge A(\tau) \right)_{t \geq 0}
\end{align*}
is an $(\mathcal{F}_{A^{-1}(t)})_{t \geq 0}$ martingale.
\\

In particular, for all $t \in \RR_{+}$, we have
\begin{align*}
	& \EE\left(1_{t \geq A(\tau)}  - t \wedge A(\tau) \right) = \EE\left(1_{0 \geq A(\tau)}  - 0 \wedge A(\tau) \right) = 0 \\
	\Longrightarrow \quad & \PP\left(A(\tau) \leq t \right) = \EE\left( t \wedge A(\tau) \right)
\end{align*}
where we use the fact that $A(\cdot)$ is strictly increasing almost surely with $A(0) = 0$, and thus $\PP(A(\tau) = 0) = 0$. Taking an increasing sequence $(t_{n})_{n \in \NN}$ such that $ t_{n} \xrightarrow{n \to \infty} t$, we can see that, by monotone convergence, 
\begin{align*}
	\PP\left( A(\tau) < t \right) = \lim_{n \to \infty} \PP(A(\tau) \leq t_{n}) = \lim_{n \to \infty} \EE \left( t_{n} \wedge A(\tau) \right) = \EE\left( t \wedge A(\tau) \right) = \PP\left( A(\tau) \leq t \right)
\end{align*}
which shows that $A(\tau)$ has a diffuse distribution.
\\

Now, if we write $Z := A(\tau)$, then
\begin{align*}
	\PP\left(Z \leq t\right) & = \EE\left(Z \wedge t \right) = \int_{0}^{\infty} \PP(Z \wedge t > s) ds = \int_{0}^{t} \PP(Z \wedge t > s) ds \\	& = t\PP(Z > t) + \int_{0}^{t} \PP(s < Z \leq t)ds =  t\PP(Z > t) + \int_{0}^{t} \PP(Z \leq t) - \PP(Z \leq s)ds \\
	& = t - t\PP(Z \leq t) + t\PP(Z \leq t) - \int_{0}^{t}\PP(Z \leq s)ds = t - \int_{0}^{t}\PP(Z \leq s)ds
\end{align*}
and thus the function $F(t) := \int_{0}^{t} \PP(Z \leq s)ds$ satisfies the ODE
\begin{align*}
	F'(t) = t - F(t) \text{  ;  } F(0) = 0 
\end{align*}
for which the continuity of $\PP(Z \leq t)$ is paramount; solved by
\begin{align*}
	F(t) = t - 1 + e^{-t}
\end{align*}
which gives
\begin{align*}
	P(A(\tau) \leq t) = F'(t) = 1 - e^{-t}
\end{align*}
and thus $A(\tau)$ is exponentially distributed with parameter 1.
\\

For the general case, when $\left(A(t)\right)_{t \geq 0}$ is not strictly increasing, we can still define its generalized inverse by using the standard definition we use when inverting continuous probability distributions. For $s \geq 0$, define
\begin{align*}
	A^{-1}(s) = \inf \left\{t \geq 0 \text{  :  } A(t) \geq s \right\}
\end{align*}
then $A^{-1}(\cdot)$ is clearly increasing and satisfies $A(A^{-1}(t)) = t$ for $t$ in the codomain of $A$, by continuity. The only thing left to prove is that $A^{-1}(A(\tau)) = \tau$ a.s., which is not direct, and is necessary for the almost sure equalities
\begin{align}
		1_{s \geq A(\tau)} & = 1_{A^{-1}(s) \geq \tau} \nonumber \\
		s \wedge A(\tau) & = A(A^{-1}(s) \wedge \tau) \label{nonstrictequalities}
\end{align} 
to hold, together with $\PP(A(\tau) = 0) = 0$; key pieces of the argument of \eqref{martingaleprop}. Note that
\begin{align*}
	A^{-1}(A(\tau)) = \inf \left\{ t \geq 0 \text{ : } A(t) \geq A(\tau) \right\}
\end{align*}
and so $\tau \geq A^{-1}(A(\tau))$. For equality to hold, it suffices to show that
\begin{align}
	\PP \left( A^{-1}(A(\tau)) < \tau \right) = 0 \label{inverseequality}
\end{align}

To prove this, we have to resort to the "well-known" section theorems and the $\sigma$-fields related to those results. For an extensive account of all the definitions and theorems needed, see \cite{Dellacherie.72}. Specifically, we note that the set $B \subseteq \RR_{+} \times \Omega$ defined by
\begin{align*}
	B = \left\{ (t,\omega) \text{ : } A^{-1}_{(\omega)}\left( A_{(\omega)} \left( \tau(\omega) \right) \right) \leq t < \tau(\omega)  \right\},
\end{align*}
where we write all the dependences on $\omega$ explicitly, is an optional set, and it satisfies that its projection onto $\Omega$ is
\begin{align*}
	\pi_{\Omega}(B) = \left\{ A^{-1}(A(\tau)) < \tau  \right\} 
\end{align*}

Now, assume
\begin{align*}
	 \delta = \PP \left( A^{-1}(A(\tau)) < \tau \right) > 0
\end{align*}
Then, applying the respective section theorem (see \cite{Dellacherie.72}, IV, T10), we know that there exists a stopping time $R$, such that
\begin{align}
	 \left\{ (t,\omega) \text{ : } t = R(\omega)  \right\} \subseteq B \label{stinclusion}
\end{align}
and also satisfies
\begin{align*}
	 \PP \left( A^{-1}(A(\tau)) < \tau \right) = \PP\left(\pi_{\Omega}(B) \right) \leq \PP(R < \infty) + \frac{\delta}{2}
\end{align*}
which gives us in particular that
\begin{align}
	\PP(R < \infty) \geq \frac{\delta}{2} > 0 \label{rpositive}
\end{align}

However, we also know that, because $R$ is a stopping time, then the martingale property of \eqref{compdef2}, together with a possible application of the monotone convergence theorem, allows us to conclude that
\begin{align*}
	\EE \left( 1_{R \geq \tau} - A\left(R \wedge \tau\right) \right) = 0
\end{align*}
which implies
\begin{align*}
	& 0 = \EE \left( 1_{R \geq \tau} - A\left(R \wedge \tau\right) \right) = \EE \left( \left(1_{R < \infty} + 1_{R = \infty} \right) \left( 1_{R \geq \tau} - A\left(R \wedge \tau\right) \right) \right) \\
	\Longrightarrow & \quad \EE\left( A\left( R\wedge \tau \right)1_{R = \infty} + A\left( R\wedge \tau \right)1_{R < \infty} \right) = \EE\left( 1_{R < \infty}1_{R \geq \tau} \right) + \EE\left(1_{R \geq \tau} 1_{R = \infty} \right) \\
	\Longleftrightarrow & \quad \EE\left( A\left( \tau \right)1_{R = \infty} + A\left(  \tau \right)1_{R < \infty} \right) =  \EE\left( 1_{R = \infty} \right)
\end{align*}
where we use the fact that, when $R$ is finite, then by \eqref{stinclusion}
\begin{align*}
	A^{-1}\left( A \left( \tau \right) \right) \leq R < \tau
\end{align*}
and so, in particular $R < \tau$ and, by applying $A(\cdot)$ to the above inequality, we recover that $A(R) = A(\tau)$.
\\

All of this amounts to the fact that
\begin{align*}
	\PP(R = \infty) = \EE(A(\tau)) = \EE\left(1_{\tau \geq  \tau}\right) = 1
\end{align*}
by again stopping our martingale \eqref{compdef2}, now at time $\tau$, which we can do because $\tau < \infty$ a.s.. This would then mean that $\PP(R < \infty) = 0$, contradicting \eqref{rpositive}, which shows \eqref{inverseequality} holds, together with the a.s. inequalities in \eqref{nonstrictequalities}. Thus, applying the same calculation we did for the strictly increasing case, all the steps follow in the same manner, and therefore $A(\tau)$ has the desired distribution.
\\

Notice that the restriction of $\tau$ being a.s. finite, which is needed to prove \eqref{inverseequality}, also eliminates a possibly problematic border case in which $\PP\left( A \equiv 0 \right) > 0$. This would render our time change argument invalid, even though \eqref{nonstrictequalities} still holds, because in particular $\PP\left( A(\tau) = 0 \right) > 0$. However, we have that
\begin{align*}
	\PP \left( A^{-1}(A(\tau)) < \tau \right) \geq \PP \left( A^{-1}(A(\tau)) < \tau, A \equiv 0 \right) = \PP \left( 0 < \tau, A \equiv 0 \right) = \PP\left( A \equiv 0 \right) > 0 
\end{align*}
which contradicts \eqref{nonstrictequalities}, and so the problematic case doesn't arise when $\tau < \infty$ a.s.
\end{proof}

\begin{remarks}
\begin{itemize}
	\item[]
	\item Note that we only use the properties of the Feller process to obtain that a jump time is always totally inaccessible. This same procedure can be applied to totally inaccessible stopping times obtained from any other framework.
	\item This result can be traced back to \cite{Azema.72}, where he proved a similar result for more general random times, and more recently to \cite{Nikeghbali.06}. However, both proofs rely heavily on the behaviour of functionals and stochastic integrals of the process $A$, while ours deals directly with the process itself, showing that the time change
	\begin{align*}
		1_{A^{-1}(s) \geq \tau}  - A(A^{-1}(s) \wedge \tau) = 1_{s \geq A(\tau)}  - s \wedge A(\tau)
	\end{align*}	
	is not only valid, but somehow essential for finite totally inaccesible stopping times, thanks to \eqref{inverseequality}.
	\item A more general result for non-finite stopping times is in corollary 6 and lemma 7 of the post Compensators of Stopping Times, in the blog Almost Sure: a Random Mathematical Blog (Lowther, 2011), albeit with no mention to the connection to jumps of Markov processes or the Cox construction.
	\item The connection to the Cox construction comes naturally out of the previous result. If we naively use $A(\tau)$ as our exponential random variable to create a stopping time with compensator $\left(A(t)\right)_{t \geq 0}$, then the resulting stopping time $\tau'$ is defined by
\begin{align*}
	\tau' & := \inf \left\{ t \geq 0 \text{  :  } A(t) \geq A(\tau)  \right\} \\
	& = \inf \left\{ t \geq 0 \text{  :  } t \geq \tau  \right\} = \tau
\end{align*}
\end{itemize}
\end{remarks} 

Of course, the requirement of independence of the exponential random variable in this case is unable to be satisfied, because our totally inaccessible stopping time is already a stopping time for the filtration we are working on. Thus, this suggests that any totally inaccessible stopping time, particularly any jump time, can potentially be understood as incorporating an external exponential random variable into a previously existing more "bare bones" filtration, for which the original stopping time is no longer a stopping time. 
\\

A natural question to pose after the previous result is, now that all finite jump times of a strong Markov process, which are totally inaccessible, come from a Cox construction set up; can we find, for any given finite totally inaccessible stopping time, a Markov process, or even a Feller process, such that said stopping time is a jump time for that process. We present a preliminary answer to this question in the next result.
\\

\begin{theorem}\label{weakmarkovfortist} Let $\left( \Omega, \mathcal{F}, \left(\mathcal{F}_{t} \right)_{t \geq 0}, \PP \right)$ contain a totally inaccessible stopping time $\tau$, such that $\PP(\tau > 0) = 1$. Then the process
\begin{align}\label{timefellerprocess}
	(X_{t})_{t \geq 0} := \left( 1_{t \geq \tau} \right)_{t \geq 0}
\end{align}
is weak Markov for the completed right-continuous version of the smaller filtration $\left( \sigma(t \wedge \tau)\right)_{t \geq 0}$, which we will call $\left( \mathcal{E}_{t} \right)_{t \geq 0}$.
\end{theorem}

\begin{proof} It is enough to prove, as stated in \cite{ChungWalsh.05} (definition 1, (iib)), that for $u > t$, 
\begin{align*}
	\EE(f(X_{u})|\mathcal{E}_{t}) = \EE(f(X_{u})|X_{t})
\end{align*}
for any bounded Borel $f: \RR \rightarrow \RR$. For this we will use the following lemma (see \cite{Protter.05}, III.5, theorem 21):
\\

\begin{lemma} Let $Y \in L^{1}(\mathcal{F})$, then
\begin{align*}
	\EE(Y|\mathcal{E}_{t}) = \EE(Y|\tau)1_{t \geq \tau} + \frac{\EE(Y 1_{t < \tau})}{\PP(t < \tau)} 1_{t < \tau}
\end{align*}
\end{lemma}

Taking $Y = f(X_{u})$ above, with $u > t$, as seen in \cite{Protter.05}, III.5, Theorem 21, we have that
\begin{align*}
	\EE(f(X_{u})|\mathcal{E}_{t}) & = \EE(f(X_{u})|\tau)1_{t \geq \tau} + \frac{\EE(f(X_{u}) 1_{t < \tau})}{\PP(t < \tau)} 1_{t < \tau} \\
	& = \EE(f(1_{u \geq \tau})|\tau)1_{t \geq \tau} + \frac{\EE(f(1_{u \geq \tau}) 1_{t < \tau})}{\PP(t < \tau)} 1_{t < \tau} \\
	& = f(1_{u \geq \tau})1_{t \geq \tau} + \frac{\EE(f(1_{u \geq \tau}) 1_{t < \tau})}{\PP(t < \tau)} 1_{t < \tau} \\
	& = K_{1}1_{t \geq \tau} + K_{2} 1_{t < \tau} \\
	& = K_{1}1_{t \geq \tau} + K_{2} ( 1 - 1_{t \geq \tau}) =: g(1_{t \geq \tau}) = g(X_{t})
\end{align*}
where $K_{1} = f(1)$ and $K_{2} = \frac{\EE(f(1_{u \geq \tau}) 1_{t < \tau})}{\PP(t < \tau)}$ are constants.
\\

On the other hand, by the tower property
\begin{align*}
	\EE(f(X_{u})|X_{t}) = \EE(\EE(f(X_{u})|\mathcal{E}_{t})|X_{t}) = \EE(g(X_{t})|X_{t}) = g(X_{t})
\end{align*}
and the result holds. 
\end{proof}

\newpage

\begin{remarks}
\begin{itemize}
	\item[]
	\item This result is a weak reciprocal to \cref{meyerstheorem}, where we have constructed, for a totally inaccessible stopping time, a weak Markov process that jumps at the desired time. 
	\item The result above does not seem to use the fact that $\tau$ is totally inaccessible, and therefore we have seemingly contradicted \cref{meyerstheorem} by finding a Markov process that jumps at a predictable time. However, notice that the space where we define the process $X$, namely the shrinking of the filtration from $\mathcal{F}_{t}$ to $\mathcal{E}_{t}$ makes $\tau$ totally inaccessible where once it was predictable, and so there is no contradiction.  
\end{itemize}
\end{remarks}

However, the Markov process defined in \eqref{timefellerprocess} is simple enough that we can attempt to prove that it is not only weak Markov, but indeed a Feller process. Therefore, if we manage to prove that this holds, then we will have the reciprocal for theorem, because not only all jump times of a Feller process will come from a Cox construction, but all totally inaccessible stopping times, in particular those Cox constructed, will indeed be jump times of a Feller process. This is our next result.
\\

\begin{theorem}\label{fellerfortist} $\left( X_{t} \right)_{t \geq 0}$ defined in \eqref{timefellerprocess} is a Feller process.
\end{theorem}

\begin{proof} Lets check each property. For any $f \in \mathcal{C}_{0}$, we have that:
\begin{enumerate}[1.]
	\item $P_{0}f(x) = \EE(f(X_{0} + x)) = \EE(f(0+x)) = \EE(f(x)) = f(x)$, because $X_{0} = 1_{0 \geq \tau} = 0$, as $\PP(\tau > 0) = 1$.
	\item $P_{t}f(x) = \EE(f(X_{t} + x)) = \EE(f(1_{t \geq \tau} + x)) = f(x)\PP(t < \tau) + f(x+1) \PP(\tau \leq t) $, which is clearly continuous in $x$ and vanishing at infinity, because $f$ does.
	\item For $x \in \RR$, we have that $f(1_{t \geq \tau} + x) \leq \max \left\{f(x),f(x+1) \right\} < \infty$. Thus, by dominated convergence and the continuity of $f$, we have that
	\begin{align*}
		\lim_{t \downarrow 0} P_{t}f(x) & = \lim_{t \downarrow 0} \EE(f(X_{t} + x)) = \lim_{t \downarrow 0} \EE(f(1_{t \geq \tau} + x)) \\
		&  = \EE(\lim_{t \downarrow 0}f(1_{t \geq \tau} + x)) = \EE(\lim_{t \downarrow 0}f(1_{t \geq \tau} + x)) = \EE(f(1_{0 \geq \tau} + x)) = \EE(f(x)) = f(x)
	\end{align*}
	because of the right continuity of $\left( 1_{t \geq \tau}\right)_{t \geq 0}$.
\end{enumerate}

Thus we have shown that $\left( 1_{t \geq \tau}\right)_{t \geq 0}$ is Feller.
\end{proof}

The previous results prove that every positive totally inaccessible stopping time is in fact the jump time of a strong Markov (particularly Feller) process. Having this result sparks even more potential connections when we notice that totally inaccessible stopping times may arise by projecting a predictable time into a smaller filtration. We saw a basic example of this in the remarks after \cref{weakmarkovfortist}, however, a more extensive example is \cite{Protter.15}, where the author describes a method of obtaining strict local martingales with jumps by projecting a continuous strict local martingale onto a smaller filtration. 
\\

This simple fact led us to ask for an analogous result to the previous ones, but for predictable stopping times. Luckily, a question posed by Monique Jeanblanc gave us the answer, which we present in the next result.
\\

\begin{theorem}\label{hittingprocesspst} Let $\tau$ be a predictable stopping time with respect to a filtration $\left( \mathcal{F}_{t} \right)_{t \geq 0}$. Then there exists a continuous process $(Y_{t})_{t \geq 0}$, adapted to a possibly larger filtration $\left( \mathcal{G}_{t} \right)_{t \geq 0}$, such that
\begin{align*}
	\tau = \inf \left\{t \geq 0 \text{ : } Y_{t} = 0 \right\}
\end{align*}
\end{theorem}

\begin{proof} Because $\tau$ is predictable, there exists a sequence of stopping times $(\tau_{n})_{n \in \NN}$, such that
\begin{enumerate}[1.]
	\item \label{pstone} For all $n \in \NN$, $\tau_{n} \leq \tau_{n+1}$. 
	\item \label{psttwo} For all $n \in \NN$, $\tau_{n} < \tau$, when $\tau > 0$. 
	\item \label{pstthree} $\tau_{n} \xrightarrow{n \to \infty} \tau$ a.s.
\end{enumerate}

Let us construct the process $Y$. Firstly, when $\tau = 0$, then $Y_{t} \equiv 0$, and we trivially satisfy what we want. When $\tau > 0$, we first focus on the case where $(\tau_{n})_{n \in \NN}$ satisfies $\tau_{n} < \tau_{n+1}$, for all $n \in \NN$. In this case, we can define the decreasing process
\begin{align*}
	Y_{t} = \sum_{i=1}^{\infty} 1_{[\tau_{i-1}, \tau_{i})}(t) \left( \frac{1}{i} - \left(  \frac{1}{i} -  \frac{1}{i+1} \right) \frac{t - \tau_{i-1}}{\tau_{i} - \tau_{i-1}} \right)
\end{align*}
where we take the convention of $\tau_{0} = 0$. Then $Y$ satisfies the following:
\begin{enumerate}[(i)]
	\item $Y_{t} > 0$ for all $t < \tau$ : if $t < \tau$, then, by condition \ref{psttwo}, $\exists n \in \NN$ such that $t \in [\tau_{n-1}, \tau_{n})$, and therefore
	\begin{align*}
		Y_{t} = \left( \frac{1}{n} + \left( \frac{1}{n+1} - \frac{1}{n} \right) \frac{t - \tau_{n-1}}{\tau_{n} - \tau_{n-1}} \right) \geq \left( \frac{1}{n} + \left( \frac{1}{n+1} - \frac{1}{n} \right) \frac{\tau_{n} - \tau_{n-1}}{\tau_{n} - \tau_{n-1}} \right) = \frac{1}{n+1} > 0 
	\end{align*}
	\item $Y_{t} = 0$ for all $t \geq \tau$: If $t \geq \tau$, then  $t \notin [\tau_{i-1}, \tau_{i})$ for any $i$, therefore $Y_{t} = 0$. 
	\item $(Y_{t})_{t \geq 0}$ is continuous: Obviously, $Y$ is right-continuous. If $t \in (\tau_{i-1}, \tau_{i})$, for some $i$, then $Y$ is linear in a neighbourhood of $t$, and so, in particular it is continuous. Now, checking continuity for $t = \tau_{i}$, for some $i$, we see that $Y_{\tau_{i}} = \frac{1}{i+1}$, while
	\begin{align*}
		Y_{\tau_{i}^{-}} = \lim_{t \uparrow \tau_{i}}Y_{t} = \lim_{t \uparrow \tau_{i}} \frac{1}{i} - \left(  \frac{1}{i} -  \frac{1}{i+1} \right) \frac{t - \tau_{i-1}}{\tau_{i} - \tau_{i-1}} = \frac{1}{i} - \left(  \frac{1}{i} -  \frac{1}{i+1} \right) = \frac{1}{i+1}
	\end{align*}
	and thus the process is continuous at $t = \tau_{i}$. For $t > \tau$, $Y_{t} = 0$ in a neighbourhood of $t$ and so continuity is trivial, so it only remains to check continuity for $t = \tau$.
\\

For this, we take a sequence of real numbers $(t_{n})_{n \in \NN}$, such that $t_{n} < \tau$ for all $n \in \NN$, and $t_{n} \xrightarrow{n \to \infty} \tau$. Then, for any fixed $n \in \NN$, there exists an $m_{0}(n) \in \NN$, such that, for all $m \geq m_{0}(n)$, we have that $t_{m} > \tau_{n}$. Thus, because $Y$ is decreasing and positive, we have that  for any $\eps > 0$, we find $n \in \NN$ such that $\frac{1}{n+1} < \eps$, and for all $m \geq m_{0}(n)$, we have
	\begin{align*}
		0 \leq Y_{t_{m}}	\leq Y_{\tau_{n}} = \frac{1}{n+1} < \eps
	\end{align*}
	which proves that $Y_{t_{n}} \xrightarrow{n \to \infty} 0 = Y_{\tau}$, which proves the continuity of $Y$ at $\tau$. 
\end{enumerate}

These results show that $(Y_{t})_{t \geq 0}$ satisfies all the path properties we want. However, one key complication is that $Y_{t}$ may not be $\mathcal{F}_{t}$ measurable, because the random variable $\frac{t - \tau_{i-1}}{\tau_{i} - \tau_{i-1}}$ is not necessarily in $\mathcal{F}_{t}$. We can remedy this by building a larger filtration $\mathcal{G} = (\mathcal{G}_{t})_{t \geq 0}$ that makes $Y_{t}$  $\mathcal{G}_{t}$-measurable. To this end, define
\begin{align*}
	\mathcal{G}_{t} = \mathcal{F}_{t} \vee \sigma(\tau_{i} - \tau_{i-1}) \quad \forall t \in [\tau_{i-1}, \tau_{i})
\end{align*}
which now satisfies that $Y_{t}$ is $\mathcal{G}_{t}$-measurable and, because $\mathcal{F}_{t} \subseteq \mathcal{G}_{t}$, $\tau$ is still predictable and it is announced by the same sequence $(\tau_{n})_{n \in \NN}$. This process and filtration prove the statement of the theorem.
\\

For the general case, we define the following subsequence $(n_{k})_{k \in \NN}$. We set $n_{0} = 0$, $n_{1} = 1$, and then
\begin{align*}
	n_{i} := \inf \left\{ k > n_{i-1} \text{ : } \tau_{n_{i}} > \tau_{n_{i-1}} \right\}
\end{align*}
which, because of properties \ref{psttwo} and \ref{pstthree}, exists and satisfies
\begin{align*}
	\tau_{n_{k}} \xrightarrow{k \to \infty} \tau
\end{align*}

Then, we reduce the general case to the previous case where $\tau_{n-1} <  \tau_{n}$ for all $n$, by considering the subsequence $(\tau_{n_{k}})_{k \in \NN}$ that also announces $\tau$, and defining in this case
\begin{align*}
	Y_{t} = \sum_{i=1}^{\infty} 1_{[\tau_{n_{i-1}}, \tau_{n_{i}})}(t) \left( \frac{1}{i} - \left(  \frac{1}{i} -  \frac{1}{i+1} \right) \frac{t - \tau_{n_{i-1}}}{\tau_{n_{i}} - \tau_{n_{i-1}}} \right)
\end{align*}
which is still $\mathcal{G}_{t}$-measurable, and satisfies all the conditions of the previously defined $Y$.
\end{proof}

\begin{remarks} 
\begin{itemize}
	\item[]
	\item We can see this last result as a parallel to the previous ones regarding totally inaccessible stopping times, in the sense that we can describe predictable stopping times as hitting times of adapted continuous processes, while totally inaccessible stopping times are jump times of discontinuous adapted processes, specifically Markov. However, it is easy to see from our construction that $Y$ is not Markov for either $\left(\mathcal{F}_{t} \right)_{t \geq 0}$ or $\left(\mathcal{G}_{t} \right)_{t \geq 0}$, and so finding a process that satisfies the conditions of \cref{hittingprocesspst}, but is also Markov, is still an open problem.
	\item The enlargement of the filtration we use for the construction of the process $Y$ is not as extensive as it looks, because even though we are adding a random variable to each interval to make everything measurable, we see that  $(\tau_{i} - \tau_{i-1}) \in \mathcal{F}_{\tau_{i}}$, and so in reality we are only performing an initial enlargement with one specific random variable at each step of enlargement, which is the next time in the predicting sequence of $\tau$. Note as well that this also shows that we are not adding $\tau$ itself in any step, which would indeed make that problem trivial. 
	\item Although we are again not aware of any similar result in the literature, there is a different proof of this result in Lemma 3 of the post Predictable Stopping Times in the blog Almost Sure: a Random Mathematical Blog (Lowther, 2009). The proof presented there has the advantage that the continuous process they construct is adapted to the original filtration of the stopping times, without the need to define $\left(\mathcal{G}_{t} \right)_{t \geq 0}$; nonetheless, it is a less explicit construction and again it is not Markov, so the open question presented above remains. 
\end{itemize}
\end{remarks}

\begingroup
	\large
		 \textbf{Statements and Declarations}
\endgroup
\\

Philip Protter is partially funded by the NSF grant DMS-2106433, together with his salary as a professor at Columbia University. Andr\'es Riveros V. is a PhD student at Columbia University, receiving a scholarship and a salary. Neither author has any non-financial conflicts of interest.

\bibliographystyle{abbrv}
\bibliography{stochfin}

\end{document}